\documentclass[11pt]{amsart}
\usepackage{amsmath, amsthm}
\usepackage{amssymb}
\usepackage{amsfonts}
\usepackage{abstract}
\usepackage{yhmath}
\usepackage{epsfig}
\usepackage{indentfirst}
\usepackage[usenames,dvipsnames]{pstricks}
\usepackage{pst-grad}
\usepackage{pst-plot}
\usepackage{mathrsfs}
\usepackage[all]{xy}
\usepackage{ragged2e}
\usepackage{bbm}

\title[Covers and Envelopes of Modules]{Existence of covers and envelopes of a left orthogonal class and its right orthogonal class of modules}

\author[Umamaheswaran, Udhayakumar, Selvaraj]{Umamaheswaran Arunachalam, Udhayakumar Ramalingam, Selvaraj Chelliah}

\address{Department of Mathematics, Harish-Chandra Research Institute (HRI), Allahabad}
\email{ruthreswaran@gmail.com}
\address{Department of Mathematics, 
Bannari Amman Institute of Technology (BIT),
Sathyamangalam, Erode - 638401, Tamil Nadu, India
}
\email{udhayaram.v@gmail.com}
\address{Department of Mathematics, Periyar University, Salem-636011, 
Tamil Nadu, India}
\email{selvavlr@yahoo.com}


\newtheoremstyle{theorem}%name
  {10pt}		  % space above
  {10pt}  % space below
  {\sl}  % bofy font
  {\parindent}     % ident - empty=no indent,  \parindent= paragraph indent
  {\bf}  % thm head font
  {. }    % punctuation after thm head
  { }    % space after thm head: `` ``=normal \newline=linebreak
  {}     % thm head specification
\theoremstyle{theorem}
\newtheorem{theorem}{Theorem}
\newtheorem{corollary}[theorem]{Corollary}

\newtheorem{proposition}[theorem]{Proposition}
\newtheorem{remark}[theorem]{Remark}

\newtheorem{example}[theorem]{Example}

\DeclareMathOperator{\ima}{im}

\newtheoremstyle{defi}%name
  {10pt}		  % space above
  {10pt}  % space below
  {\rm}  % bofy font
  {\parindent}     % ident - empty=no indent,  \parindent= paragraph indent
  {\bf}  % thm head font
  {. }    % punctuation after thm head
  { }    % space after thm head: `` ``=normal \newline=linebreak
  {}     % thm head specification
\theoremstyle{defi}
\newtheorem{definition}[theorem]{Definition}

\newtheoremstyle{defi}%name
{10pt} % space above
{10pt} % space below
{\rm} % bofy font
{\parindent} % ident - empty=no indent, \parindent= paragraph indent
{\bf} % thm head font
{. } % punctuation after thm head
{ } % space after thm head: `` ``=normal \newline=linebreak
{} % thm head specification
\theoremstyle{defi}

\begin{document}

\maketitle

\begin{abstract}
In this paper, we investigate the notions of $\mathcal{X}^\bot$-projective, $\mathcal{X}$-injective and $\mathcal{X}$-flat modules and give some characterizations of these modules, where $\mathcal{X}$ is a class of left $R$-modules. We prove that the class of all $\mathcal{X}^\bot$-projective modules is Kaplansky. Further, if the class of all $\mathcal{X}$-projective $R$-modules is closed under direct limits, we show the existence of $\mathcal{X}^\bot$-projective covers  and $\mathcal{X}$-injective envelopes over a $\mathcal{X}^\bot$-hereditary ring $R.$ Moreover, we decompose a $\mathcal{X}^\bot$-projective module into a projective and a coreduced $\mathcal{X}^\bot$-projective module over a self $\mathcal{X}$-injective and $\mathcal{X}^\bot$-hereditary ring. Finally, we prove that every module has a $\mathcal{W}$-injective precover over a coherent ring $R,$ where $\mathcal{W}$ is the class of all pure projective modules.
\end{abstract}

\section{Introduction}
 Throughout this paper, $R$ denotes an associative ring with identity and all $R$-modules, if not specified otherwise, are left $R$-modules. $R$-$Mod$ denotes the category of left $R$-modules.

	The notion of $FP$-injective modules over arbitrary rings was first introduced by Stenstr\"om in \cite{stns}. An $R$-module $M$ is called $FP$-injective if $Ext^{1}_{R}(N,M) = 0$ for all finitely presented $R$-modules $N$. Let $\mathcal{X}$ be a class of left $R$-modules. Mao and Ding in \cite{mao2} introduced the concept of $\mathcal{X}$-injective modules (see Definition \ref{2.0}).
 
	The notions of (pre)covers and (pre)envelopes of modules were introduced by Enochs in \cite{Eno1} and, independently, by Auslander and Smal\o\, in \cite{aus}. Since then the existence and the properties of (pre)covers and (pre)envelopes relative to certain submodule categories have been studied widely. The theory of (pre)covers and (pre)envelopes, which play an important role in homological algebra and representation theory of algebras, is now one of the main research topics in relative homological algebra.

	Salce introduced the notion of a cotorsion theory in \cite{salce}. Enochs showed the important fact that closed and complete cotorsion pairs provide minimal versions of covers and envelopes. Eklof and Trlifaj \cite{eklof} proved that a cotorsion pair $(\mathcal{A}, \mathcal{B})$ is complete when it is cogenerated by a set. Consequently, many classical cotorsion pairs are complete. In this way, Enochs et al. \cite{bashir} showed that every module has a flat cover over an arbitrary ring. These results motivate us to prove the existence of $\mathcal{X}^\bot$-projective cover and $\mathcal{X}$-injective envelope. In particular, we prove the following result.
	
%\begin{proposition}\label{0.1}
%Let $R$ be a $\mathcal{X}^\bot$-hereditary ring. Then the class $\widetilde{\mathcal{X}^\bot}$ is Kaplansky.
%\end{proposition}

\begin{theorem}\label{p.2}
Let $R$ be a $\mathcal{X}^\bot$-hereditary ring. If the class of all $\mathcal{X}$-projective $R$-modules is closed under direct limits, then every $R$-module $M$ has a $\mathcal{X}^\bot$-projective cover and an $\mathcal{X}$-injective envelope.
\end{theorem}

	Self injective rings were introduced by Johnson and Wong in \cite{john}. A ring $R$ is said to be self injective if $R$ over itself is an injective module. In this paper, we introduce self $\mathcal{X}$-injective ring (see Definition \ref{self}). Mao and Ding \cite{mao3} proved that an $FI$-injective $R$-module decomposes into an injective and a reduced $FI$-injective $R$-module over a coherent ring. Similarly, we can prove the following result:
	
\begin{theorem}\label{p.3}
Let $R$ be a self $\mathcal{X}$-injective and $\mathcal{X}^\bot$-hereditary ring. If the class of all $\mathcal{X}$-projective $R$-modules is closed under direct limits, then an $R$-module $M$ is $\mathcal{X}^\bot$-projective if and only if $M$ is a direct sum of a projective $R$-module and a coreduced $\mathcal{X}^\bot$-projective $R$-module.
\end{theorem}
	
	Enochs and Pinzon \cite{pin} proved that every module has an $FP$-injective cover over a coherent ring. We prove the following result that provides the existence of $\mathcal{W}$-injective cover, where $\mathcal{W}$ is the class of all pure projective modules.
	
\begin{theorem}\label{p.4}
Let $R$ be a coherent ring. Then every $R$-module $M$ has a $\mathcal{W}$-injective cover.
\end{theorem}
  
	This paper is organized as follows: In Section 2, we recall some notions that are necessary for our proofs of the main results of this paper.

	In Section 3, we investigate the notions of $\mathcal{X}$-injective and $\mathcal{X}$-flat modules and give some characterizations of these classes of modules.
 
	In Section 4, we introduce $\mathcal{X}^\bot$-hereditary ring. Further, we investigate $\mathcal{X}^\bot$-projective module and give some characterizations. Further, we prove Theorem \ref{p.2} and Theorem \ref{p.3}.
	
	In Section 5, we prove that if $M$ is a submodule of an $\mathcal{X}$-injective $R$-module $A$, then $i \colon M \rightarrow A$ is a special $\mathcal{X}$-injective envelope of $M$ if and only if $A$ is an $\mathcal{X}^\bot$-projective essential extension of $M.$
	
	In the last section, we assume that $\mathcal{W}$ is the class of all pure projective modules and we prove that every module has a $\mathcal{W}$-injective  preenvelope. Moreover, we prove Theorem \ref{p.4}.
 
\section{preliminaries}

	In this section, we recall some known definitions and some terminology that will be used in the rest of the paper.

	Given a class $\mathscr{C}$ of left $R$-modules, we write
\begin{eqnarray*}
\mathscr{C}^{\bot} &=& \left\{N \,\in R\mbox{-}Mod \, | \ Ext_{R}^{1}(M, N) = 0 \, \mbox{ $\forall$ } M \in \mathscr{C}\right\}\\
^{\bot}\mathscr{C} &=& \left\{N \,\in R\mbox{-}Mod \, | \ Ext_{R}^{1}(N, M) = 0 \, \mbox{ $\forall$ } M \in \mathscr{C}\right\}.
\end{eqnarray*}

	Following \cite{Eno1}, we say that a map $f \in Hom_{R}(C, M)$ with $C \in \mathscr{C}$ is a $\mathscr{C}$-precover of $M$, if the group homomorphism $Hom_{R}(C', f)$ $\colon$ $Hom_{R}(C', C) \rightarrow Hom_{R}(C', M)$ is surjective for each $C'$ $\in \mathscr{C}$. A $\mathscr{C}$-precover $f \in Hom_{R}(C, M)$ of $M$ is called a $\mathscr{C}$-cover of $M$ if $f$ is right minimal. That is, if $fg = f$ implies that $g$ is an automorphism for each $g \in$ $End_{R}(C)$. $\mathscr{C} \subseteq$ $R\mbox{-}Mod$ is a precovering class (resp. covering class) provided that each module has a $\mathscr{C}$-precover (resp. $\mathscr{C}$-cover). Dually, we have the definition of $\mathscr{C}\mbox{-}$preenvelope (resp. $\mathscr{C}\mbox{-}$envelope).
 
 A $\mathscr{C}$-precover $f$ of $M$ is said to be \textit{special} \cite{tri} if $f$ is an epimorphism and $ker f \in\, \mathscr{C}^\bot$. 

 A $\mathscr{C}$-preenvelope $f$ of $M$ is said to be \textit{special} \cite{tri} if $f$ is a monomorphism and $coker f \in\, ^\bot\mathscr{C}$.
 
 A $\mathscr{C}$-envelope $\phi \colon M \rightarrow C$ is said to have the \textit{unique mapping property} \cite{Ding} if for any homomorphism $f \colon M \rightarrow C'$ with $C' \in \mathscr{C}$, there is a unique homomorphism $g \colon C \rightarrow C'$ such that $g\phi = f.$
 
 Recall that an $R$-module $M$ is called \textit{reduced} \cite{enochs} if it has no nonzero injective submodules. An $R$-module $M$ is said to be \textit{coreduced} \cite{chen} if it has no nonzero projective quotient modules.

	A module is said to be \textit{pure projective} \cite{pp} if it is projective with respect to pure exact sequence.
 
 A class $\mathscr{C}$ of left $R$-modules is said to be \textit{injectively resolving} \cite{tri} if $\mathscr{C}$ contains all injective modules and if given an exact sequence of left $R$-modules
\begin{center}
$0 \rightarrow A \rightarrow B \rightarrow C\rightarrow 0$
\end{center}
$C \in \mathscr{C}$ whenever $A, B \in \mathscr{C}.$

 \begin{definition}\label{0.2}
\begin{enumerate}
\item A pair $\mathfrak{C} = \left(\mathcal{A},\,\mathcal{B}\right)$ of classes of modules is called a \textit{cotorsion} \cite{tri} if $\mathcal{A} =\,^\bot\mathcal{B}$ and $\mathcal{B} = \mathcal{A}^{\bot}.$

\item A cotorsion theory $(\mathcal{A}, \mathcal{B})$ is said to be perfect \cite{lop} if every module has an $\mathcal{A}$-cover and a $\mathcal{B}$-envelope.
\end{enumerate}
\end{definition}

 For an $R$-module $M$, $fd(M)$ denote the flat dimension of $M$ and $id(M)$ denote the injective dimension of $M.$ 
 
 The $\mathcal{X}^\bot$-coresolution dimension of $M$, denoted by $cores.dim_{\mathcal{X}^\bot}(M)$, is defined to be the smallest nonnegative integer $n$ such that $Ext_{R}^{n + 1}(A, M) = 0$ for all $R$-modules $A \in \mathcal{X}$ (if no such $n$ exists, set  $cores.dim_{\mathcal{X}^\bot}(M) = \infty$), and  $cores.dim_{\mathcal{X}^\bot}(R)$ is defined as $\sup \{cores.dim_{\mathcal{X}^\bot}(M) | M \in R$-$Mod\}.$
 
 We denote by $\mathbb{Z}$ the ring of all integers, and by $\mathbb{Q}$ the field of all rational numbers. For a left $R$-module $M$, we denote by $M^{+} = Hom_{\mathbb{Z}}(M, \mathbb{Q}/ \mathbb{Z})$ the \textit{character module} of $M.$ $\mathcal{I}_0$ denotes the class of all injective left $R$-modules.

 For unexplained terminology we refer to \cite{and, rot}.

\section{$\mathcal{X}$-injective and $\mathcal{X}$-flat modules}

	We begin with the following definition:
	
\begin{definition}\label{2.0}\cite{mao2}
 A left $R$-module $M$ is called $\mathcal{X}$-injective if $Ext_{R}^{1}(X, M) = 0$ for all left $R$-modules $X \in \mathcal{X}.$
 A right $R$-module $N$ is said to be $\mathcal{X}$-flat if $Tor^{R}_{1}(N, X) = 0$ for all left $R$-modules $X \in \mathcal{X}.$
\end{definition}

Also, we denote by $\mathcal{X}^\bot$ the class of all $\mathcal{X}$-injective modules and $^\bot(\mathcal{X}^\bot)$ the class of all $\mathcal{X}^\bot$-projective $R$-modules.

% \begin{remark}
%  It is clear that the class of all injective modules lies between the class of all $\mathcal{X}$-injective modules and the class of all flat module lies between the class of all $\mathcal{X}$-flat modules.
% \end{remark}

\begin{proposition}\label{2.3} Let $M$ be an $R$-module. Then the following are true.
\begin{enumerate}
\item Let $\mathcal{I}_0 \subseteq \mathcal{X}.$ Then $M$ is injective if and only if $M$ is $\mathcal{X}$-injective and $id(M) \leq 1.$
\item $M$ is injective if and only if $M$ is $\mathcal{X}^\bot$-injective and $cores.dim_{\mathcal{X}^\bot}(M) \leq 1.$
\end{enumerate}
\end{proposition}

\begin{proof}
$(1).$ The direct implication is clear. Conversely, let $M$ be $\mathcal{X}$-injective and $id(M) \leq 1.$ For any $R$-module $N,$ consider an exact sequence $0 \rightarrow N \rightarrow E(N) \rightarrow L \rightarrow 0$ with $E(N)$ an injective envelope of $N.$ We have an exact sequence \[\cdots \rightarrow Ext^{1}_{R}(E(N), M) \rightarrow Ext^{1}_{R}(N, M) \rightarrow Ext^{2}_{R}(L, M) \rightarrow \cdots.\] Since $id(M) \leq 1$, $Ext_{R}^{2}(L, M) = 0$ and hence $M$ is injective.

$(2).$ The direct implication is clear by the definition of an injective module. Conversely, let $M$ be $\mathcal{X}^\bot$-injective and $cores.dim_{\mathcal{X}^\bot}(M) \leq 1$. Consider an exact sequence $0 \rightarrow M \rightarrow E(M) \rightarrow L \rightarrow 0$ with $E(M)$ an injective envelope of $M.$ For any $R$-module $X \in \mathcal{X},$  we have an exact sequence $\cdots \rightarrow Ext^{1}_{R}(X,E(M)) \rightarrow Ext^{1}_{R}(X, L) \rightarrow Ext^{2}_{R}(X, M) \rightarrow \cdots$. Since $cores.dim_{\mathcal{X}^\bot}(M) \leq 1$, $Ext_{R}^{2}(X, M) = 0$ and hence $L$ is $\mathcal{X}$-injective. Therefore,  $Ext^{1}_{R}(L, M) = 0$, so that the exact sequence is split. It follows that $M$ is a direct summand of $E,$ as desired.
\end{proof}

We now give some of the characterizations of $\mathcal{X}$-injective module:
\begin{proposition}\label{2.5}
Let $\mathcal{I}_0 \subseteq \mathcal{X}.$The following are equivalent for a left $R$-module $M:$
\begin{enumerate}
  \item [(1)] $M$ is $\mathcal{X}$-injective;
  \item [(2)] For every exact sequence $0 \rightarrow M \rightarrow E \rightarrow L \rightarrow 0$, with $E \in \mathcal{X},$ $E \rightarrow L$ $\mathcal{X}$-precover of $L;$
  \item [(3)] $M$ is a kernel of an $\mathcal{X}$-precover $f \colon A \rightarrow B$ with $A$ an injective module;
  \item [(4)] $M$ is injective with respect to every exact sequence $0 \rightarrow A \rightarrow B \rightarrow C \rightarrow 0$, with $C \in \mathcal{X}.$
\end{enumerate}
\end{proposition}

\begin{proof} 
\par $(1) \Rightarrow (2)$. Consider an exact sequence 
\begin{center}
$0 \rightarrow M \rightarrow E \rightarrow L \rightarrow 0$,
\end{center}
where $E \in \mathcal{X}.$ Then by hypothesis $Hom_{R}(E', E) \rightarrow Hom_{R}(E', L)$ is surjective for all left $R$-modules $E' \in \mathcal{X},$  as desired.
\par $(2) \Rightarrow (3)$. Let $E(M)$ be an injective hull of $M$ and consider the exact sequence $0 \rightarrow M \rightarrow E(M) \stackrel{f}{\rightarrow} E(M)/M \rightarrow 0.$ Since $E(M)$ is injective, it belongs to $\mathcal{X}.$ So assertions $(3)$ holds.
\par $(3) \Rightarrow (1)$. Let $M$ be a kernel of an $\mathcal{X}$-precover $f \colon A \rightarrow B$ with $A$ an injective module. Then we have an exact sequence $0 \rightarrow M \rightarrow A \rightarrow A/M \rightarrow 0$. Therefore, for any left $R$-module $N \in \mathcal{X},$ the sequence $Hom_{R}(N, A)$ $\rightarrow Hom_{R}(N, A/M) \rightarrow Ext_{R}^{1}(N, M) \rightarrow 0$ is exact. By hypothesis, $Hom_{R}(N, A) \rightarrow Hom_{R}(N, A/M)$ is surjective. Thus $Ext_{R}^{1}(N, M) = 0$, and hence (1) follows.
\par $(1) \Rightarrow (4)$. Consider an exact sequence
\begin{center}
$0 \rightarrow A \rightarrow B \rightarrow C \rightarrow 0$,
\end{center}
where $C \in \mathcal{X}.$ Then $Hom_{R}(B, M) \rightarrow Hom_{R}(A, M)$ is surjective, as desired.
\par $(4) \Rightarrow (1)$. For each left $R$-module $N \in \mathcal{X},$ there exists a short exact sequence $0 \rightarrow K \rightarrow P \rightarrow N \rightarrow 0$ with $P$ a projective module, which induces an exact sequence $Hom_{R}(P, M) \rightarrow Hom_{R}(K, M) \rightarrow Ext_{R}^{1}(N, M) \rightarrow 0$. By hypothesis, $Hom_{R}(P, M) \rightarrow Hom_{R}(K, M) \rightarrow 0$ is exact. Thus $Ext_{R}^{1}(N, M) = 0$, and hence (1) follows.
\end{proof}

\begin{example}\label{2.13}
Let $(R, \mathfrak{m})$ be a commutative Noetherian and complete local domain. Assume that the $depth R \geq 2$ and $cores.dim_{\mathcal{X}^\bot}(R) \leq 1$. Then $R/\mathfrak{m} \oplus E(R)$ is an  $(\mathcal{X}^\bot)^\bot$-injective module.
\end{example}

\begin{proof}
Consider the residue field $k = R/\mathfrak{m}$ and an exact sequence $0 \rightarrow k \rightarrow E(k) \stackrel{\phi}{\rightarrow} E(k)/k \rightarrow 0$. If $G$ is an $R$-module, the sequence $Hom_{R}(G, E(k)) \rightarrow Hom_{R}(G, E(k)/k) \rightarrow Ext_{R}^{1}(G, k) \rightarrow 0$ is exact. By \cite[p ~43]{xu}, $\phi$ is an injective cover of $E(k)/k.$ Since $cores.dim_{\mathcal{X}^\bot}(R) \leq 1,$ then the class of all injective $R$-modules and the class of all $\mathcal{X}^\perp$-injective $R$-modules are equal, that is, $\mathcal{I}_0 = (\mathcal{X}^\perp)^\perp$. Clearly, $\phi$ is an $\mathcal{X}^\bot$-injective cover of $E(k)/k$. Thus $Hom_{R}(G^\prime, E(k)) \rightarrow Hom_{R}(G^\prime, E(k)/k)$ is surjective for every $\mathcal{X}^\perp$-injective $R$-module $G^\prime$. So $Ext_{R}^{1}(G^\prime, k) = 0$ for every $\mathcal{X}^\perp$-injective $R$-module $G^\prime$, and hence $k$ is $(\mathcal{X}^\bot)^\bot$-injective. On the other hand $E(R)$ is injective and so $(\mathcal{X}^\perp)^\perp$-injective. Therefore $k \oplus E(R)$ is $(\mathcal{X}^\perp)^\perp$-injective.
\end{proof}

\begin{example}\label{2.15}
Let $(R, \mathfrak{m})$ be a commutative Noetherian and complete local ring. Assume that the $depth R \leq 1$. Then the residue field $k = R/\mathfrak{m}$ is not $(\mathcal{X}^\bot)^\bot$-injective.
\end{example}

\begin{proof}
Consider an exact sequence $0 \rightarrow k \rightarrow E(k) \stackrel{\phi}{\rightarrow} E(k)/k \rightarrow 0$. Hence by \cite[Corollary ~5.4.7]{enochs}, $\phi$ is not an injective cover of $E(k)/k$. This implies that $\phi$ is not an $\mathcal{X}^\bot$-injective cover of $E(k)/k$. Then by Proposition \ref{2.5}, $k$ is not $(\mathcal{X}^\bot)^\bot$-injective.
\end{proof}

We now give some characterizations of $\mathcal{X}$-flat module:

\begin{proposition}\label{2.16}
The following are equivalent for a right $R$-module $M$:
\begin{enumerate}
\item $M$ is $\mathcal{X}$-flat;
\item For every exact sequence $0 \rightarrow A \rightarrow B \rightarrow C \rightarrow 0$ with $C \in \mathcal{X},$ the functor $M \otimes_R -$ preserves the exactness;
\item $Ext_{R}^{1}(M, G^{+}) = 0$ for all $G \in \mathcal{X};$
\item $M^{+}$ is $\mathcal{X}$-injective. 
\end{enumerate}
\end{proposition}

\begin{proof}
$(1) \Rightarrow (2).$ Consider an exact sequence  $0 \rightarrow A \rightarrow B \rightarrow C \rightarrow 0$ with $C \in \mathcal{X}.$ Since $M$ is $\mathcal{X}$-flat, $Tor^{R}_{1}(M, C) = 0$. Hence the functor $M \otimes -$ preserves the exactness.

$(2) \Rightarrow (1).$ Let $G \in \mathcal{X}.$ Then there exists a short sequence $0 \rightarrow K \rightarrow F \rightarrow G \rightarrow 0$ with $F$ a projective module, which induces an exact sequence $0 \rightarrow Tor^{R}_{1}(M, G) \rightarrow M \otimes K \rightarrow M \otimes F \rightarrow M \otimes G \rightarrow 0$. By hypothesis, $Tor^{R}_{1}(M, G) = 0$. Thus $M$ is $\mathcal{X}$-flat.

$(1)\Leftrightarrow(3).$ It follows from the natural isomorphism \cite[p 34]{Fuchs} $Tor^{R}_{1}(M, G)^{+} \cong Ext_{R}^{1}(M, G^{+})$.

$(1)\Leftrightarrow(4).$ It follows from the natural isomorphism \cite[VI 5.1]{car} $Ext_{R}^{1}(G, M^{+}) \cong Tor^{R}_{1}(M, G)^{+}.$
\end{proof}

\begin{proposition}\label{2.4}
Let $R$ be a coherent ring. Then a right $R$-module $N$ is flat if and only if $N$ is $\mathcal{X}$-flat and $fd(N) \leq 1$.
\end{proposition}

\begin{proof} ``only if" part is trivial. Conversely, suppose that $N$ is $\mathcal{X}$-flat. By Remark \ref{2.16}, $N^{+}$ is $\mathcal{X}$-injective. By \cite[Theorem ~2.1]{field}, $fd(N) = id(N^{+})$. Then $cores.dim_{\mathcal{X}^\bot}(N^{+}) \leq 1$ since $cores.dim_{\mathcal{X}^\bot}(N^{+}) \leq id(N^{+})$. By Proposition \ref{2.3}, $N^{+}$ is injective and hence $N$ is flat.
\end{proof}

%%%%%%%%%%%%%%%%%%%%%%%%%%%%%%%%%%%%%%%%%%%%%%%%%%%%%%%%%%%%%%%%%%%%%%%%%%%%%%%%%%%%%%%%%%%%%%%%%%%%%%%%%%%%%%%%%%%%%%%%%%%%%%%%%%%%%%%%%%%%%%%%%%%%%%%%%%%%%%%%%%%%%%%%%%%%%%%%%%%%%%%%%%%%%%%%%%%%%%%%%%%%%%%%%%%%%%%%%%%%%%%%%%%%%%%%%%%%%%%%%%%%%%%%%%%%%%%%%%%%%%%%%%%%%%%%%%%%%%%%%%%%%%%%%%%%%%%%%%%%%%%%%%%%%%%%%%%%%%%%%%%%%%%%%%%%%%%%%%%%%%%%%%%%%%%%%%%%%%%%%%%
\section{$\mathcal{X}^\bot$-projective cover and $\mathcal{X}$-injective envelope}

Now, we introduce $\mathcal{X}$-projective module.

\begin{definition}
An $R$-module $M$ is called a $\mathcal{X}$-projective if $Ext_{R}^{1}(M, X) = 0$ for all $R$-modules $X \in \mathcal{X}$.
\end{definition}

\begin{example}\label{2.1.0}
Let $(R, \mathfrak{m})$ be a complete local ring. Assume that the $cores.dim_{\mathcal{X}^\bot}(R) \newline \leq 1$. Then $E(k)/k$ is $(\mathcal{X}^\bot)^\bot$-projective, where $k = R/\mathfrak{m}$ is the residue field and $E(M)$ is $\mathcal{X}^\bot$-injective envelope of $k$.
\end{example}

\begin{proof}
Consider an exact sequence $0 \rightarrow k \rightarrow E(k) \rightarrow E(k)/k \rightarrow 0,$ where $E()$ is an injective envelope of $k.$ Since $cores.dim_{\mathcal{X}^\bot}(R) \leq 1,$ then $\mathcal{I}_0 = (\mathcal{X}^\perp)^\perp$. It follows that $E(k)$ is an $\mathcal{X}^\perp$-injective envelope of $k.$ Since the class of all $(\mathcal{X}^\bot)^\bot$-projective modules is closed under extensions, then by \cite [Lemma ~2.1.2]{xu} $E(k)/k$ is $(\mathcal{X}^\bot)^\bot$-projective.
\end{proof}

We now introduce the following definition

\begin{definition}
A ring $R$ is called a left $\mathcal{X}$-hereditary if every left ideal of $R$ is $\mathcal{X}$-projective.
\end{definition}

% \begin{example}
% Semisimple ring, Dedekind domain and hereditary ring. 
% \end{example}

\begin{remark}\label{2.1.2}
If $\mathcal{X}$ is the class of injective left $R$-modules, then every ring is $\mathcal{X}$-hereditary. It is also easy to see that a ring is left hereditary if and only if $R$ is $\mathcal{X}$-hereditary for every class $\mathcal{X}$ of left $R$-modules.
\end{remark}

 Given a class $\mathcal{X}$ of left $R$-modules, we denote by $\mathcal{X}^{\bot_2}$ the class \[\{N \in R\mbox{-}Mod\,|\, Ext_R^2(X, N) = 0 \mbox{ for every } X \in \mathcal{X}\}.\]

\begin{example}\label{2.1.2.1}
Let $R$ be a commutative Noetherian ring. If $\mathcal{X} = \{R / \mathfrak{p} \colon \mathfrak{p} \in Spec R\},$ then $R$ is a $\mathcal{X}^{\bot_2}$-hereditary ring.
\end{example}

\begin{proof}
Let $I$ be an ideal of $R$. We claim that $I$ is $\mathcal{X}^{\bot_2}$-projective. By hypothesis, $Ext_R^2(R/\mathfrak{p}, G) = 0$ for all $\mathfrak{p} \in Spec R$ and for all $G \in \mathcal{X}^{\bot_2}.$ It follows that $id(G) \leq 1.$ Thus  $Ext_R^2(R/I, G) = 0$ for all ideals $I$ of $R$  and for all $G \in \mathcal{X}^{\bot_2}.$ Consider an exact sequence $0 \rightarrow I \rightarrow R \rightarrow R/I \rightarrow 0.$ So $Ext_R^1(I, G) = 0$ for all $G \in \mathcal{X}^{\bot_2}.$ Hence $I$ is $\mathcal{X}^{\bot_2}$-projective, as desired.
\end{proof}

Note that an $R$-module $M$ is called a $\mathcal{X}^\bot$-projective if $Ext_{R}^{1}(M, U) = 0$ for all $R$-modules $U \in \mathcal{X}^\bot$. Clearly, $(^\bot(\mathcal{X}^\bot), \mathcal{X}^\bot)$ is a cotorsion theory.

\begin{proposition}\label{2.1.3}
A ring $R$ is left $\mathcal{X}^\bot$-hereditary if and only if every submodule of a $\mathcal{X}^\bot$-projective left $R$-module is $\mathcal{X}^\bot$-projective.
\end{proposition}

\begin{proof}
Let $R$ be left $\mathcal{X}^\bot$-hereditary and $I$ be a left ideal of $R$. Then there is an exact sequence $0 \rightarrow I \rightarrow R \rightarrow R/I \rightarrow 0.$ By hypothesis, $Ext_R^{2}(R/I, U) = 0$ for any $\mathcal{X}$-injective left $R$-module $U.$ Thus $id(U) \leq 1.$ Let $G$ be a submodule of an $\mathcal{X}^\bot$-projective left $R$-module $H.$ Then, for any $\mathcal{X}$-injective left $R$-module $U$, the sequence $\cdots \rightarrow Ext_{R}^1(H, U) \rightarrow Ext_R^1(G, U) \rightarrow Ext_R^2(H/G, U) \rightarrow \cdots$ is exact. Thus $Ext_R^1(G, U) = 0$ since $H$ is $\mathcal{X}^\bot$-projective and $id(U) \leq 1.$ The reverse implication is clear.
\end{proof}

% \begin{example}
% Let $R$ be a commutative Noetherian ring and $\mathcal{X} = \{R / \mathfrak{p} \colon \mathfrak{p} \in Spec R\}.$ Let $S = \mathbb{M}_n(R).$ Then $S$ is a $\mathcal{X}^{\bot_2}$-hereditary ring.
% \end{example}

% \begin{proof}
% By Example \ref{2.1.2.1}, $R$ is $\mathcal{X}^{\bot_2}$-hereditary. For $S = \mathbb{M}_n(R),$ the module categories $_{R}\mathfrak{C}$ and $_{S}\mathfrak{D}$ are equivalent by \cite[Theorem ~17.20]{lam}. From this natural eqivalance, $\mathcal{X}^{\bot_2}$ projective $S$-modules $M$ correspond to $\mathcal{X}^{\bot_2}$ projective $R$-modules $N.$ Hence submodules of $M$ correspond to submodules of $N.$ By Proposition \ref{2.1.3}, all submodules of $N$ are $\mathcal{X}^{\bot_2}$ projective in $_{R}\mathfrak{C}.$ It follows that submodules of $M$ are $\mathcal{X}^{\bot_2}$-projective in $_{S}\mathfrak{D}$. This implies that $S$ is a $\mathcal{X}^{\bot_2}$-hereditary ring by Proposition \ref{2.1.3}.
% \end{proof}

% In general, the class $^\bot(\mathcal{X}^\bot)$ of all $\mathcal{X}^\bot$-projective modules is not necessarily closed under pure submodules. If $R$ is a hereditary ring, then $^\bot(\mathcal{X}^\bot) = R$-$Mod$ and $\mathcal{X}^\bot = \mathcal{I}_0$. It follows that $^\bot(\mathcal{X}^\bot)$ is closed under pure submodules. If $R$ is $\mathcal{X}^\bot$-hereditary, then $\mathcal{X}^\bot$ is need not be equal to $\mathcal{I}_0$. Is it true that $^\bot(\mathcal{X}^\bot)$ a closed under pure submodules over a $\mathcal{X}^\bot$-hereditary ring? We have the following

In general $^\bot(\mathcal{X}^\bot)$ is not closed under pure submodules, for example if $\mathcal{F}$ is the class of flat modules, then $^\bot(\mathcal{F}^\bot) = \mathcal{F}$ and this class is not closed submodules in general. As a consequence of Proposition \ref{2.1.3} we have

\begin{corollary}\label{1.1}
Let $R$ be a $\mathcal{X}^\bot$-hereditary ring. Then the class $^\bot(\mathcal{X}^\bot)$ is closed under pure submodules.
\end{corollary}

% \begin{proof}
% It follows from Proposition \ref{2.1.3}.
% \end{proof}

\begin{definition}\cite{lop}
Let $\mathcal{K}$ be a class of $R$-modules. Then $\mathcal{K}$ is said to be \textit{Kaplansky class} if there exists a cardinal $\aleph$ such that for every $M \in \mathcal{K}$ and for each $x \in M$, there exists a submodule $K$ of $M$ such that $x \in K \subseteq M, K, M/K \in \mathcal{K}$ and $Card(K) \leq \aleph.$
\end{definition}

\begin{proposition}\label{1.2}
Let $R$ be a $\mathcal{X}^\bot$-hereditary ring. Then $^\bot(\mathcal{X}^\bot)$ is Kaplansky.
\end{proposition}

\begin{proof}
Let $M \in\,^\bot(\mathcal{X}^\bot)$ and $x \in M$. Consider an exact sequence
\begin{center}
$\mathbb{M}_\bullet \colon 0 \rightarrow K \rightarrow P \stackrel{f}{\rightarrow} M \rightarrow 0$
\end{center}
with $P$ a projective module and by Proposition \ref{2.1.3} $K$ is $\mathcal{X}^\bot$-projective and we will construct a $\mathcal{X}^\bot$-projective submodule $F$ of $M$ with $x \in F$. 

 Since $x \in M$ and $f$ is surjective, there exists $y \in M$ such that $f(y) = x.$ Consider  $<y> \rightarrow S$ the inclusion and we get by \cite[Lemma ~5.3.12]{enochs}, a cardinal $\aleph_{\alpha}$ and a submodule $F \subseteq P$ pure such that $<y> \subseteq F$ and $Card (F) \leq \aleph_{\alpha}$. 
 
 Since $f$ is surjective, then there exists a submodule $F_0$ of $M$ which applies in the preceding module. Let $f(F_0) \subseteq M$ and $F_1 \subseteq K$ which applies in $f(F_0)$. We obtain $F_1^{\prime} \subseteq K$ pure and $\aleph_{\beta}$ such that $F_1 \subseteq F_1^{\prime}$ and $Card (F_1^{\prime}) \leq \aleph_{\beta}.$ Then we can get the following exact sequence
 \begin{center}
$\mathbb{F}_\bullet \colon 0 \rightarrow F_1 \rightarrow F_0 \rightarrow F \rightarrow 0.$
 \end{center}
Since $F_1$ and $F_0$ are pure submodules of $\mathcal{X}^\bot$-projective modules, then by Proposition \ref{1.1} $F_1$ and $F_0$ are $\mathcal{X}^\bot$-projective.

Finally, $M/F \in\,^\bot(\mathcal{X}^\bot)$ since the quotient of exact sequences $\mathbb{M}_\bullet / \mathbb{F}_\bullet$ is exact and it remains exact when $Hom_R(-, G)$ is applied for any $\mathcal{X}^\bot$-projective $R$-module $G$ because $\mathbb{M}_\bullet$
 and $\mathbb{S}_\bullet$ verify the two conditions.
\end{proof}

\begin{theorem}\label{1.3}
Let $R$ be a $\mathcal{X}^\bot$-hereditary ring. If the class of all $\mathcal{X}$-projective $R$-modules is closed under direct limits, then every $R$-module $M$ has a $\mathcal{X}^\bot$-projective cover and an $\mathcal{X}$-injective envelope.
\end{theorem}

\begin{proof}
By Proposition \ref{1.2}, $^\bot(\mathcal{X}^\bot)$ is a Kaplansky class. Since all projective modules are $\mathcal{X}^\bot$-projective, $^\bot(\mathcal{X}^\bot)$ contains the projective modules. Clearly, $^\bot(\mathcal{X}^\bot)$ is closed under extensions. By hypothesis, $^\bot(\mathcal{X}^\bot)$ is closed under direct limits. Then by \cite[Theorem ~2.9]{lop}, $(^\bot(\mathcal{X}^\bot), \mathcal{X}^\bot)$ is a perfect cotorsion theory. Hence by Definition \ref{0.2}, every module has a $^\bot(\mathcal{X}^\bot)$-cover and a $\mathcal{X}^\bot$-envelope.
\end{proof}

Now we introduce self $\mathcal{X}$-injective ring.

\begin{definition}\label{self}
 A ring $R$ is said to be self $\mathcal{X}$-injective if $R$ over itself is an $\mathcal{X}$-injective module. 
\end{definition}

We now give some characterizations of $\mathcal{X}^\bot$-projective module:
 
\begin{proposition}\label{4.8}
Let $R$ be a self $\mathcal{X}$-injective ring and let $M$ be an $R$-module. Then the following conditions are equivalent:
\begin{enumerate}
\item $M$ is $\mathcal{X}^\bot$-projective;
\item $M$ is projective with respect to every exact sequence $0 \rightarrow A \rightarrow B \rightarrow C \rightarrow 0$, with $A$ an $\mathcal{X}$-injective module;
\item For every exact sequence $0 \rightarrow K \rightarrow P \rightarrow M \rightarrow 0$, where $P$ is $\mathcal{X}$-injective, $K \rightarrow P$ is an $\mathcal{X}$-injective preenvelope of $K;$
\item $M$ is cokernel of an $\mathcal{X}$-injective preenvelope $K \rightarrow P$ with $P$ a projective module.
\end{enumerate} 
\end{proposition}

\begin{proof}
$(1) \Rightarrow(2)$. Let $0 \rightarrow A \rightarrow B \rightarrow C \rightarrow 0$ be an exact sequence, where $A$ is $\mathcal{X}$-injective. Then by hypothesis $Hom_{R}(M, B) \rightarrow Hom_{R}(M, C)$ is surjective.
\par $(2) \Rightarrow(1)$. Let $N$ be an $\mathcal{X}$-injective $R$-module. Then there is a sequence $0 \rightarrow N \rightarrow E \rightarrow L \rightarrow 0$ with $E$ an injective envelope of $N$. By (2), $Hom_{R}(M, E) \rightarrow Hom_{R}(M, L)$ is surjective. Thus $Ext_{R}^{1}(M, N) = 0$, as desired.
\par $(1) \Rightarrow(3)$. Clearly, $Ext_{R}^{1}(M, F') = 0$ for all $\mathcal{X}$-injective $F'$. Hence we have an exact sequence $Hom_{R}(F, F') \rightarrow Hom_{R}(K, F') \rightarrow 0$.
\par $(3) \Rightarrow(4)$. Consider an exact sequence $0 \rightarrow K \rightarrow P \rightarrow M \rightarrow 0$ where $P$ projective. Since $R$ is self $\mathcal{X}$-injective, every projective module is $\mathcal{X}$-injective. Hence $P$ is $\mathcal{X}$-injective. Then by hypothesis $K \rightarrow P$ is an $\mathcal{X}$-injective preenvelope.
\par $(4) \Rightarrow(1)$. By hypothesis, there is an exact sequence $0 \rightarrow K \rightarrow P \rightarrow M \rightarrow 0$, where $K \rightarrow P$ is an $\mathcal{X}$-injective preenvelope with $P$ projective. It gives rise to the exactness of $Hom_{R}(P, N) \rightarrow Hom_{R}(K, N) \rightarrow Ext_{R}^{1}(M, N) \rightarrow 0$ for each $\mathcal{X}$-injective $R$-module $N$. Since $R$ is self $\mathcal{X}$-injective, $Hom_{R}(P, N) \rightarrow Hom_{R}(K, N)$ is surjective. Hence $Ext_{R}^{1}(M, N) = 0$, as desired.
\end{proof}

\begin{proposition}\label{4.11}
Let $R$ be a self $\mathcal{X}$-injective. If the class of all $\mathcal{X}$-projective $R$-modules is closed under direct limits, then the following are equivalent for an $R$-module $M$:
\begin{enumerate}
 \item [(1)] $M$ is coreduced $\mathcal{X}^\bot$-projective;
 \item [(2)] $M$ is a cokernel of an $\mathcal{X}$-injective envelope $K \rightarrow P$ with $P$ a projective module.
\end{enumerate}
\end{proposition}

\begin{proof}
$(1) \Rightarrow (2)$. Consider an exact sequence $0 \rightarrow K \stackrel{f}{\rightarrow} P \stackrel{g}{\rightarrow} M \rightarrow 0$ with $P$ a projective module. Since $R$ is self $\mathcal{X}$-injective, $P$ is $\mathcal{X}$-injective. By Proposition \ref{4.8}, the natural map $f \colon K \rightarrow P$ is an $\mathcal{X}$-injective preenvelope of $K$. By Proposition \ref{1.3}, $K$ has an $\mathcal{X}$-injective envelope $\alpha \colon K \rightarrow P^{'}$. Then there exist $\beta \colon P^{'} \rightarrow P$ and $\beta^{'} \colon P \rightarrow P^{'}$ such that $\alpha = \beta^{'}f$ and $f = \beta \alpha$. Hence $\alpha = (\beta \beta^{'})\alpha$. It follows that $\beta \beta^{'}$ is an isomorphism, $P = \ima (\beta) \oplus \ker (\beta^{'})$. Note that $\ima (f) \subseteq \ima (\beta)$, and so $P/\ima (f) \rightarrow P/\ima (\beta) \rightarrow 0$ is exact. But $M$ is coreduced and $P/\ima(f) \cong M$, and hence $P/\ima(\beta) = 0$, that is, $P = \ima(\beta).$ So $\beta$ is an isomorphism, and hence $f \colon K \rightarrow P$ is an $\mathcal{X}$-injective envelope of $K$.

\par $(2) \Rightarrow (1)$. By Proposition \ref{4.8}, $M$ is $\mathcal{X}^\bot$-projective and $M$ is coreduced by \cite[Lemma ~3.7]{Ding}
\end{proof}

We are now to prove the main result of this section.
\begin{theorem}
Let $R$ be a self $\mathcal{X}$-injective ring and the class of all $\mathcal{X}$-projective $R$-modules is closed under direct limits. Then an $R$-module $M$ is $\mathcal{X}^\bot$-projective if and only if $M$ is a direct sum of a projective $R$-module and a coreduced $\mathcal{X}^\bot$-projective $R$-module.
\end{theorem}

\begin{proof}
``If'' part is clear.

``Only if'' part. Let $M$ be a $\mathcal{X}^\bot$-projective $R$-module. By Proposition \ref{4.8}, we have an exact sequence $0 \rightarrow K \rightarrow P \rightarrow M \rightarrow 0$ with $P$ a projective module, where $K \rightarrow P$ is an $\mathcal{X}$-injective preenvelope of $K$ . By Proposition \ref{1.3}, $K$ has an $\mathcal{X}$-injective envelope $f \in Hom_R(K, P^{'})$ with $P^\prime$ an $\mathcal{X}$-injective $R$-module. Then we have the following commutative diagram with exact rows:
\begin{center}
\[\xymatrix@C-.15pc@R-.18pc{
0 \ar[r]  & K \ar[r] \ar@{=}[d]  & P^{'} \ar[r] \ar[d]^\alpha & P^{'}/\ima f  \ar[d]^\phi \ar[r]&0& \\ 
0 \ar[r]  & K \ar[r] \ar@{=}[d]  & P \ar[d]^\beta \ar[r] & M \ar[d]^\sigma\ar[r] &0&\\
0 \ar[r]  & K \ar[r]^f  & P^{'} \ar[r] & P^{'}/\ima f \ar[r] &0&\\
}\]
\end{center}
Note that $\beta\alpha$ is an isomorphism, and so $P = \ker \beta \oplus \ima \alpha$. Since $\ima \alpha \cong P^{'}$, $P^{'}$ and $\ker \beta$ are projective. Therefore $P^{'}/\ima f$ is a coreduced $\mathcal{X}^\bot$-projective module by Proposition \ref{4.11}. By the Five Lemma, $\sigma\phi$ is an isomorphism. Hence, we have $M = \ima \phi \oplus \ker \sigma$, where $\ima \phi \cong P^{'}/\ima f$. In addition, we get the following commutative diagram:
\begin{center}
\[\xymatrix@C-.15pc@R-.18pc{
& 0\ar[d]  & 0\ar[d]& 0 \ar[d]& &\\
0\ar[r] & 0 \ar[r] \ar[d] & \ker \beta \ar[r] \ar[d]& \ker \sigma \ar[r] \ar[d]& 0& \\ 
0 \ar[r]  & K \ar@{=}[d] \ar[r] & P \ar[d]^\beta \ar[r] & M \ar[d]^\sigma \ar[r] &0&\\
0 \ar[r]  & K \ar[d] \ar[r]  & L \ar[d] \ar[r] & P^{'}/\ima f \ar[d] \ar[r] &0&\\
& 0& 0& 0.& &
}\]
\end{center}
Hence, $\ker \sigma \cong \ker \beta$.
\end{proof}
%%%%%%%%%%%%%%%%%%%%%%%%%%%%%%%%%%%%%%%%%%%%%%%%%%%%%%%%%%%%%%%%%%%%%%%%%%%%%%%%%%%%%%%%%%%%%%%%%%%%%%%%%%%%%%%%%%%%%%%%%%%%%%%%%%%%%%%%%%%%%%%%%%%%%%%%%%%%%%%%%%%%%%%%%%%%%%%%%%%%%%%%%%%%%%%%%%%%%%%%%%%%%%%%%%%%%%%%%%%%%%%%%%%%%%%%%%%%%%%%%%%%%%%%%%%%%%%%%%%%%%%%%%%%%%%%%%%%%%%%%%%%%%%%%%%%%%%%%%%%%%%%%%%%%%%%%%%%%%%%%%%%%%%%%%%%%%%%%%%%%%%%%%%%%%%%%%%%%%%%%%%

\section{Some Relation Between $\mathcal{X}^\bot$-projective and $\mathcal{X}$-injective modules}

 In this section, we deals with $\mathcal{X}$-injective envelope of a module and $\mathcal{X}^\bot$-projective module.
 
%\begin{definition}
%An $R$-module $M$ is called the $\mathcal{X}^\bot$-projective if $Ext_{R}^{1} \left(M, N\right) = 0$ for all $\mathcal{X}$-injective $R$-modules $N$. In other words, $\widetilde{\mathcal{V}_\mathcal{X}} = \,^{\bot}\mathcal{X}^\bot$.
%\end{definition}
%
%\begin{example}
%Let $R$ be a commutative Noetherian ring and $\mathfrak{p}$ be a prime ideal of $R$. Assume that $cores.dim_{\mathcal{X}^\bot}(R) \leq 1$ and $E(R/\mathfrak{p})$ is the injective envelope of $R/\mathfrak{p}$. Then $\bigoplus E_{i}(R_{i}/\mathfrak{p}_{i}) /\bigoplus R_{i}/\mathfrak{p}_{i}$ is $\mathcal{X}^\bot$-projective, where $R_{i} = R$, $\mathfrak{p}_{i} =\mathfrak{p}$, $E_{i}(R_{i}/\mathfrak{p}_{i}) = E(R/\mathfrak{p})$ for all $i \geq 1$.
%\end{example}
%
%\begin{lemma}\label{4.3}
 %$\mathcal{X}^\bot = \widetilde{\mathcal{V}_\mathcal{X}}^{\bot}$, i.e., $\mathcal{X}^\bot = \left(^{\bot}\mathcal{X}^\bot\right)^{\bot}$.
%\end{lemma}
%
%\begin{proof} 
%Clearly, $\mathcal{X}^\bot \subseteq (^{\bot}$$\mathcal{X}^\bot)^{\bot}$. On the other hand, let $M \in (^{\bot}$$\mathcal{X}^\bot)^{\bot}$ and let $X$ be a $\mathcal{X}^\bot$-projective $R$-module. Then $X \in$ $^{\bot}$$\mathcal{X}^\bot$. Therefore $Ext_{R}^{1}(M, X) = 0$, and so $M \in \mathcal{X}^\bot$.
%\end{proof}
%
%\begin{remark}\label{4.4}
%By Lemma \ref{4.3}, $\mathfrak{C} = \left(\widetilde{\mathcal{V}_\mathcal{X}}, \mathcal{X}^\bot\right)$ is a cotorsion theory. If $R$ is a $\mathcal{X}^\bot$-hereditary ring, then by Proposition \ref{4.12} $\mathfrak{C}$ is a complete cotorsion theory.
%\end{remark}

\begin{theorem}\label{4.5}
Let $\phi$ $\colon M \rightarrow A$ be an $\mathcal{X}$-injective envelope. Then $L = A/\phi(M)$ is $\mathcal{X}^\bot$-projective and hence $A$ is $\mathcal{X}^\bot$-projective whenever $M$ is $\mathcal{X}^\bot$-projective.
\end{theorem}
\begin{proof}
It follows from \cite[Lemma ~2.1.2]{xu}.
\end{proof}

\begin{theorem}\label{4.7}
Let $0 \rightarrow M \rightarrow A \rightarrow D \rightarrow 0$ be a minimal generator of all $\mathcal{X}^\bot$-projective extensions of $M$. Then $A$ is an $\mathcal{X}$-injective envelope of $M$.
\end{theorem}

\begin{proof}
It follows from \cite[Theorem ~2.2.1]{xu}.
\end{proof}
 
 Let $M$ be a submodule of a module $A$. Then $A$ is called a $\mathcal{X}^\bot$-projective extension of a submodule $M$ if $A/M$ is $\mathcal{X}^\bot$-projective.
 
 Recall that among all $\mathcal{X}^\bot$-projective extensions of $M$ we call one of them $0 \rightarrow M \rightarrow A \rightarrow D \rightarrow 0$ a generator for $\mathcal{E}$$\textit{xt}$$\left(^\bot(\mathcal{X}^\bot), M\right)$ (or a generator for all $\mathcal{X}^\bot$-projective extensions of $M$) if for any $\mathcal{X}^\bot$-projective extension 
$0 \rightarrow M \rightarrow A' \rightarrow D' \rightarrow 0$ of $M$, then there is a commutative diagram
\begin{center}
\[\xymatrix@C-.15pc@R-.18pc{
0\ar[r] & M \ar[r] \ar@{=}[d] & A' \ar[r] \ar[d]& D' \ar[r] \ar[d]& 0& \\ 
0 \ar[r]  & M \ar[r] & A \ar[r] & D\ar[r] &0.&\\
}\]
\end{center}
Furthermore, a generator $0 \rightarrow M \rightarrow A \rightarrow D \rightarrow 0$ is called minimal if for all the vertical maps are isomorphisms whenever $A', D'$ are replaced by $A$, $D$, respectively.
  
\begin{theorem}\label{4.6}
Let $R$ be a $\mathcal{X}^\bot$-hereditary ring and the class of all $\mathcal{X}$-projective $R$-modules is closed under direct limits. Then for an $R$-module $M$, there must be a minimal generator whenever $\mathcal{E}\textit{xt}$$\left(^\bot(\mathcal{X}^\bot), M\right)$ has a generator.
\end{theorem} 

\begin{proof}
It follows from \cite[Theorem ~2.2.2]{xu}.
\end{proof}

\begin{theorem}\label{4.9}
Suppose that an $R$-module $M$ has an $\mathcal{X}$-injective envelope. Let $M$ be a submodule of an $\mathcal{X}$-injective $R$-module $L$. Then the following are equivalent:
\begin{enumerate}
 \item [(1)] $i \colon M \rightarrow L$ is a special $\mathcal{X}$-injective envelope;
 \item [(2)] $L/M$ is $\mathcal{X}^\bot$-projective, and there are no direct summands $L_{1}$ of $L$ with $L_{1} \neq L$ and $M \subseteq L_{1}$;
 \item [(3)] $L/M$ is $\mathcal{X}^\bot$-projective, and for any epimorphism $\alpha \colon L/M \rightarrow N$ such that $\alpha\pi$ is split, $N = 0$, where $\pi \colon L \rightarrow L/M$ is the canonical map;
 \item [(4)] $L/M$ is $\mathcal{X}^\bot$-projective, and any endomorphism $\gamma$ of $L$ such that $\gamma i = i$ is a monomorphism;
 \item [(5)] $L/M$ is $\mathcal{X}^\bot$-projective, and there is no nonzero submodule $N$ of $L$ such that $M \cap N = 0$ and $L=(M \oplus N)$ is $\mathcal{X}^\bot$-projective.
\end{enumerate}
\end{theorem}

\begin{proof}
$(1)\Leftrightarrow(2)$ follows from \cite[Corollary ~1.2.3]{xu} and Theorem \ref{4.5}.
\par $(2) \Rightarrow (3)$. Since $\alpha\pi$ is split, there is a monomorphism $\beta : N \rightarrow L$ such that $L = \ker (\alpha\pi) \oplus \beta(N)$. Note that $M \subseteq \ker(\alpha\pi)$, and so $L = \ker (\alpha\pi)$ by (2). Thus $\beta(N) = 0$, and hence $N = 0$.
\par $(3) \Rightarrow (2)$. If $L = L_{1} \oplus N$ with $M \subseteq L_{1}$. Let $p \colon L \rightarrow N$ be a
canonical projection. Then there is an epimorphism $\alpha \colon L/M \rightarrow N$ such that $\alpha\pi = p$. Thus $N = 0$ by hypothsis, and hence $L = L_{1}$, as required.
\par $(1) \Rightarrow (4)$. By Wakamatsu's Lemma \cite[Proposition ~7.2.4]{enochs}, $L/M$ is $\mathcal{X}^\bot$-projective. Since $\gamma i = i$ and $i$ is monomorphism, $\gamma$ is monomorphism.
\par $(4) \Rightarrow (1)$. Since $L/M$ is $\mathcal{X}^\bot$-projective, $i$ is a special $\mathcal{X}$-injective preenvelope. Let $\psi \colon M \rightarrow \mathcal{X}^\bot(M)$ be an $\mathcal{X}$-injective envelope of $M$. Then there exist $\mu \colon L \rightarrow \mathcal{X}^\bot(M)$ and $\nu \colon \mathcal{X}^\bot(M) \rightarrow L$ such that $\mu i = \psi$ and $\nu\psi = i$. Hence $\mu\nu\psi = \psi$ and $i = \nu\mu i$. Thus $\mu\nu$ is an isomorphism, and so $\mu$ is epic. In addition, by (4), $\nu\mu$ is monic, and hence $\mu$ is monic. Therefore $\mu$ is an isomorphism, and hence $i$ is an $\mathcal{X}$-injective envelope of $M$.
\par $(1) \Rightarrow (5).$ It is obvious that $L/M$ is $\mathcal{X}^\bot$-projective. Suppose there is a nonzero submodule $N \subseteq L$ such that $M \cap N = 0$ and $L = (M \oplus N)$ is $\mathcal{X}^\bot$-projective. Let $\pi \colon L \rightarrow L/N$ be a canonical map. Since $L/(N \oplus M)$ is $\mathcal{X}^\bot$-projective and $L$ is $\mathcal{X}$-injective, there is a $\beta \colon L/N \rightarrow L$ such that the following diagram with row exact
\begin{center}
\[\xymatrix@C-.15pc@R-.18pc{
& & L \ar[d]^\pi&&\\
0 \ar[r]  & M \ar[ur]^i \ar[d]^i \ar[r]^\alpha & L/N \ar[r] \ar[dl]^\beta & L/(N \oplus M) \ar[r]& 0&\\
&L&&&
}\]
\end{center}
is commutative. Hence $\beta\pi i = i$. Note that $i$ is an envelope, and so $\beta\pi$ is an isomorphism, whence $\pi$ is an isomorphism. But this is impossible since $\pi(N) = 0$.
\par $(5) \Rightarrow (1)$. Let $\psi_{M} \colon M \rightarrow \mathcal{X}^\bot(M)$ be an $\mathcal{X}$-injective envelope of $M$.
Since $L/M$ is $\mathcal{X}^\bot$-projective, $i$ is a special $\mathcal{X}$-injective preenvelope. Thus we have the following commutative diagram with an exact row.
\begin{center}
\[\xymatrix@C-.15pc@R-.18pc{
0 \ar[r]  & M \ar[dr]^i \ar[r]^{\psi_{M}} & \mathcal{X}^\bot(M) \ar[r]^\phi \ar@<1ex>[d]^f & Q \ar[r]& 0&\\
& & L \ar@<1ex>[u]^g \ar[ur]^\alpha&  & &
}\]
\end{center}
i.e., $f\psi_{M} = i,$ $gi = \psi_{M}$. So $gf\psi_{M} = \psi_{M}$. Note that $\psi_{M}$ is an $\mathcal{X}$-injective
envelope, and hence $gf$ is an isomorphism. Without loss of generality, we may assume $gf = 1$. Write $\alpha = \phi g \colon L \rightarrow Q$. It is clear that $\alpha$ is epic and $M \cap \ker(g) = 0$. We show that $M \oplus \ker(g) = \ker(\alpha)$. Clearly, $M \oplus \ker(g) \subseteq \ker(\alpha)$. Let $x \in \ker(\alpha)$. Then $\alpha(x) = \phi g(x) = 0$. It follows that $g(x) = \psi_{M}(m)$ for some $m \in M$, and hence $fg(x) = f\psi_{M}(m) = m, g(x) = gfg(x) = g(m)$. Thus $x \in M \oplus \ker(g)$, and so $\ker(\alpha) \subseteq M \oplus \ker(g)$, as desired. Consequently, $L =(M \oplus \ker(g)) = L/\ker(\alpha) \cong Q$ is $\mathcal{X}^\bot$-projective by Wakamatsu's Lemma. Thus $\ker(g) = 0$ by hypothesis, and hence $g$ is an isomorphism. So $i \colon M \rightarrow L$ is an $\mathcal{X}$-injective envelope.
\end{proof}
A submodule of $\mathcal{X}$-injective module need not be $\mathcal{X}$-injective.
\begin{example}
Let $(R, \mathfrak{m})$ be a commutative Noetherian and local domain. Assume that $depth R \leq 1.$ Then the submodule $k$ of injective envelope $E(k)$ of $k$ is not $\mathcal{X}$-injective.
\end{example}

\begin{proof}
Consider an exact sequence $0 \rightarrow k \rightarrow E(k) \stackrel{\phi}{\rightarrow} E(k)/k \rightarrow 0$. Hence by \cite[Corollary ~5.4.7]{enochs}, $\phi$ is not an injective cover of $E(k)/k$. This implies that $\phi$ is not an $\mathcal{X}^\bot$-injective cover of $E(k)/k$. Then by Proposition \ref{2.5}, $k$ is not $\mathcal{X}$-injective.
\end{proof}

\begin{theorem}\label{4.10}
Let $R$ be a $\mathcal{X}^\bot$-hereditary ring and the class of all $\mathcal{X}$-projective $R$-modules is closed under direct limits. If $M$ be a submodule of an $\mathcal{X}$-injective $R$-module $A$, then the following are equivalent:
\begin{enumerate}
\item $i \colon M \rightarrow A$ is a special $\mathcal{X}$-injective envelope of $M;$
\item $A$ is a $\mathcal{X}^\bot$-projective essential extension of $M.$
\end{enumerate}
\end{theorem}

\begin{proof}
$(1) \Rightarrow (2).$ It follows by Proposition \ref{4.9}.
\par $(2) \Rightarrow (1).$ By hypothesis, we have an exact sequence: $0 \rightarrow M \rightarrow A \rightarrow L \rightarrow 0$ with $A$ an $\mathcal{X}$-injective module and $D$ an $\mathcal{X}^\bot$-projective module. This sequence is a generator of all $\mathcal{X}^\bot$-projective extensions of $M$. By Theorem \ref{4.6} and \ref{4.7}, we have an $\mathcal{X}^\bot$-projective extension sequence of $M$  $0 \rightarrow M \rightarrow A' \rightarrow L' \rightarrow 0$ which gives an $\mathcal{X}$-injective envelope of $M$. Then we have the following commutative diagram:
\begin{center}
\[\xymatrix@C-.15pc@R-.18pc{
0\ar[r] & M \ar[r]^{\alpha^{'}} \ar@{=}[d] & A' \ar[r]^{\beta^{'}} \ar@<1ex>[d]^f& L' \ar[r] \ar[d]&0 & \\ 
0 \ar[r]  & M \ar[r]^\alpha & A \ar[r]^\beta \ar@<1ex>[u]^g & L \ar[r] &0.&\\
}\]
\end{center}
It is easy to see that $A = f(A') \oplus \ker(g)$. We claim that $\ker(g) = 0$. Since $M = f\alpha^{'}(M) \subseteq f(A')$, $\ker(g) \cap M = 0$. We define the following homomorphism $\psi \colon A / (M \oplus \ker(g)) \rightarrow L'$, $a + (M \oplus \ker(g)) \mapsto \alpha^{'}g(a)$. Obviously, $\psi$ is well defined. By diagram chasing, we see that $\psi$ is injective. But both $g$ and $\beta^{'}$ are surjective,  so is $\psi$. Therefore, $\ker(g)$ is $\mathcal{X}^\bot$-projective essential extension of $A/\ker(g)$. This contradicts the hypothesis that $A$ is $\mathcal{X}^\bot$-projective essential extension of $M$. This implies that $\ker(g) = 0$ and so $f$ is an isomorphism.
\end{proof}

%%%%%%%%%%%%%%%%%%%%%%%%%%%%%%%%%%%%%%%%%%%%%%%%%%%%%%%%%%%%%%%%%%%%%%%%%%%%%%%%%%%%%%%%%%%%%%%%%%%%%%%%%%%%%%%%%%%%%%%%%%%%%%%%%%%%%%%%%%%%%%%%%%%%%%%%%%%%%%%%%%%%%%%%%%%%%%%%%%%%%
\section{$\mathcal{W}$-injective cover}

In this section, we assume $\mathcal{W}$ is the class of all pure projective modules and we prove that all modules have $\mathcal{W}$-injective covers. 

\begin{proposition}\label{6.3.1}
The class $\mathcal{W}^\bot$ of all $\mathcal{W}$-injective modules is closed under pure submodules.
\end{proposition}

\begin{proof}
Let $A$ be a pure submodule of a $\mathcal{W}$-injective module $M$. Then there is a pure exact sequence $0 \rightarrow A \rightarrow M \rightarrow M/A \rightarrow 0$ and a functor $Hom_R(X, -)$ preseves this sequence is exact whenever $X \in \mathcal{W}.$ This implies that the sequence $0 \rightarrow Hom_R(W, A) \rightarrow Hom_R(W, M) \rightarrow Hom_R(W, M/A) \rightarrow Ext_R^1(W, A) \rightarrow 0$ is also exact for all $W \in \mathcal{W}$. It follows that $Ext_R^1(W, A) = 0$ for all $W \in \mathcal{W}$, as desired.
\end{proof}

\begin{theorem}\label{6.3.2}
Every $R$-module has a $\mathcal{W}$-injective preenvelope.
\end{theorem}

\begin{proof}
Let $M$ be an $R$-module. By \cite[Lemma ~5.3.12]{enochs}, there is a cardinal number $\aleph_\alpha$ such that for any $R$-homomorphism $\phi \colon M \rightarrow G$ with $G$ a $\mathcal{W}$-injective $R$-module, there exists a pure submodule $A$ of $G$ such that $|A| \leq \aleph_\alpha$ and $\phi (M) \subset A.$ Clearly, $\mathcal{W}^\bot$ is closed under direct products and by Proposition \ref{6.3.1} $A$ is $\mathcal{W}$-injective. Hence the theorem follows by \cite[Proposition ~6.2.1]{enochs}.
\end{proof}

\begin{proposition}\label{3.3}
The class $\mathcal{W}^\bot$ of all $\mathcal{W}$-injective modules is injectively resolving.
\end{proposition}

\begin{proof}
Let $0 \rightarrow M_{1} \stackrel{\phi}{\rightarrow} M_{2} \stackrel{\psi}{\rightarrow} M_{3} \rightarrow 0$ be an exact sequence of left $R$-modules with $M_{1}, M_{2} \in  \mathcal{W}^\bot$. Let $G \in \mathcal{W}^\bot$. By Theorem \ref{6.3.2}, every module has a $\mathcal{W}^\bot$-preenvelope. By \cite[Lemma ~1.9]{tri1}, $G$ has a special $\mathcal{W}^\bot$-preenvelope. By \cite[Lemma ~2.2.6]{tri}, $G$ has a special $^\bot(\mathcal{W}^\bot)$-precover. Then there exists an exact sequence  $0 \rightarrow K \rightarrow A \rightarrow G \rightarrow 0$ with $A \in ^\bot(\mathcal{W}^\bot)$ and $K \in \mathcal{W}^\bot$. We prove that $M_{3}$ is $\mathcal{W}$-injection, i.e., to prove that $Ext_{R}^{1}(G, M_{3}) = 0$. For this it suffices to extend any $\alpha \in Hom_{R}(K, M_{3})$ to an element of $Hom_{R}(A, M_{3})$. Clearly, $K$ has $^\bot(\mathcal{W}^\bot)$-precover, 
\begin{center}
$0 \rightarrow K' \stackrel{f}{\rightarrow} A' \stackrel{g}{\rightarrow} K \rightarrow 0$,
\end{center}
where $K$, $K^{'} \in \mathcal{W}^\bot$ and $A' \in ^\bot(\mathcal{W}^\bot)$. As the class $\mathcal{W}^\bot$ is closed under extensions, $A' \in \mathcal{W}^\bot$. Since $\alpha \circ g \colon A' \rightarrow M_{3}$ with $A^{'} \in \mathcal{W}^\bot$ and $M_{1}$ a $\mathcal{W}$-injective module, then there exists $\beta \colon A' \rightarrow M_{2}$ such that $\psi \circ \beta = \alpha \circ g$. That is, the following diagram is commutative
\begin{center}
\[\xymatrix@C-.15pc@R-.18pc{
A' \ar@{-->}[d]^\beta \ar[r]^g & K \ar[d]^\alpha &\\ 
M_{2} \ar[r]^\psi & M_{3}.&\\
}\]
\end{center}
Now, we define $\beta\upharpoonright_{\ima \phi} \colon A' \rightarrow \ima \phi$, where $\upharpoonright$ is a restriction map. Then there exists $\gamma \colon K' \rightarrow M_{1}$ such that $\beta\upharpoonright_{\ima \phi} (f(K'))=\phi \gamma(K')$. Hence we have the following commutative diagram
\begin{center}
\[\xymatrix@C-.15pc@R-.18pc{
0 \ar[r] & K' \ar[r]^f \ar@{-->}[d]^\gamma & A' \ar@{-->}[d]^\beta \ar[r]^g & K  \ar[d]^\alpha \ar[r]& 0&\\ 
0 \ar[r] & M_{1} \ar[r]^\phi & M_{2} \ar[r]^\psi & M_{3} \ar[r]& 0.&\\
}\]
\end{center}
The $\mathcal{W}$-injectivity of $M_{1}$ yields a homomorphism $\gamma_{1} \colon A' \rightarrow M_{1}$ such that $\gamma = \gamma_{1} \circ f$. So for each $k' \in K'$, we get $(\beta \circ f)(k') = (\phi \circ \gamma)(k') = (\phi \circ (\gamma_{1} \circ f))(k')$. Then there exists a map $\beta_{1} \in Hom_{R}(K, M_{2})$ such that $\beta = \beta_{1} \circ g$ and we get $\alpha = \psi \circ \beta_{1}$.
Thus the following diagram is commutative
\begin{center}
\[\xymatrix@C-.15pc@R-.18pc{
0 \ar[r] & K' \ar[r]^f \ar@{-->}[d]_\gamma & A' \ar@{-->}[dl]_{\gamma_1} \ar@{-->}[d]_\beta \ar[r]^g & K  \ar@{-->}[dl]_{\beta_1} \ar[d]_\alpha \ar[r]& 0&\\ 
0 \ar[r] & M_{1} \ar[r]^\phi & M_{2} \ar[r]^\psi & M_{3} \ar[r]& 0. &\\
}\]
\end{center}
Since $M_{2}$ is $\mathcal{W}$-injective, there exists $\rho \in Hom_{R}(A, M_{2})$ such that $\beta_{1} = \rho \circ f$. Thus $\alpha = \psi \circ \beta_{1} = \psi \circ (\rho \circ f)$, where $\psi \circ \rho \in Hom_{R}(A, M_{3})$. Hence $M_{3}$ is $\mathcal{W}$-injective. 
\end{proof}

\begin{proposition}\label{3.5}
Let $M$ be a $\mathcal{W}$-injective $R$-module and $A$ be a pure submodule $M$. Then an $R$-module $M/A$ is $\mathcal{W}$-injective.
\end{proposition}

\begin{proof}
By Proposition \ref{3.3}, $\mathcal{W}^\bot$ is injectively resolving. Let $M \in \mathcal{W}^\bot$ and $A$ be a pure submodule of $M$. By Proposition \ref{6.3.1}, $A$ is $\mathcal{W}$-injective. From the short exact sequence $0 \rightarrow A \rightarrow M \rightarrow M/A \rightarrow 0$, we get $M/A$ is $\mathcal{W}$-injective.
\end{proof}

	We recall that $M$ is pure injective if and only if $M$ is a direct summand of a direct sum of finitely presented $R$-modules \cite{azumaya}. It follows that $M$ is a submodule of a finitely presented module. If $R$ is coherent, then direct sum of $\mathcal{W}$-injective $R$-modules is $\mathcal{W}$-injective.

	The following result establishes an analog version of Theorem 2.6 in \cite{pin}.

\begin{theorem}\label{3.7}
Let $R$ be a coherent ring. Then every $R$-module $M$ has a $\mathcal{W}$-injective precover.
\end{theorem}

\begin{proof}
 Let $W$ be a set with $Card(W) \leq \kappa$, where $\kappa$ is the cardinal in \cite[Theorem ~5]{bashir}. Denote $\mathcal{P}(W)$ the power set of $W$. We find all the binary operations $* \colon G \times G \rightarrow G$ for each element $G \in \mathcal{P}(W)$ and we get a new collection $\cup_{G \in \mathcal{P}(W)}$ $\left\{G, *\right\}= \overline{\mathcal{G}}^{'}$. From $\overline{\mathcal{G}}^{'}$, find all the scalar multiplications, which are functions from the cross product into itself. This remains a set $\cup_{G \in \mathcal{P}(W)}\left\{(G, *, \cdot)\right\}$ which is denoted by $\overline{\mathcal{G}}$. Some collection of members of $\overline{\mathcal{G}}$ form a module and we can get the class $\mathcal{G}$ of $\mathcal{W}$-injective modules which is contained in the class $\mathcal{W}^\bot$. Clearly $\mathcal{I}_0$ is contained in the class $\mathcal{G}$. Since $R$ is coherent, $\bigoplus_{N \in \mathcal{I}_0} N^{\left(Hom_{R}(N, M)\right)}$ is an injective module and hence it is $\mathcal{W}$-injective. We prove that $\bigoplus_{N \in \mathcal{I}_0} N^{\left(Hom_{R}(N, M)\right)} \stackrel{\phi}{\rightarrow} M$ is a $\mathcal{W}$-injective precover. That is, to show that if for any homomorphism $N' \rightarrow M$ with $N' \in \mathcal{W}^\bot$, then that the following diagram
\begin{center}
\[\xymatrix@C-.15pc@R-.18pc{
N'\ar[dr] \ar@{-->}[d] & &\\
\bigoplus\limits_{N \in \mathcal{I}_0} N^{\left(Hom_{R}(N, M)\right)} \ar[r] & M &
}\]
\end{center}
is commutative.

 Let $K$ be the kernel of the map $N' \rightarrow M$, where $N'$ is sufficiently large. Then $N'/K$ is sufficiently small since $\left| N'/K\right| \leq \left|M\right|$. By Bashir's Theorem \cite{bashir}, $K$ has a nonzero submodule $L$ that is pure in $N'$. Therefore $L$ is $\mathcal{W}$-injective by Proposition \ref{6.3.1}. This implies that $N'/L$ is $\mathcal{W}$-injective by Proposition \ref{3.5}. If $N'/L$ is still sufficiently large, then repeat the process from the map $N'/L \rightarrow M$. Since $N'/K_{1}$ is sufficiently small, $K_{1}/L$ has a nonzero submodule $L_{1}/K_{1}$ that is pure in $N'/L$. Thus $(N'/K_{1})/(L_{1}/K_{1}) = N'/L_{1}$ is $\mathcal{W}$-injective. But again, this may be too large. Then by continuing this process we get $\underrightarrow{lim}\, (N'/L_{i})$ is $\mathcal{W}$-injective since $\mathcal{W}^\bot$ is closed under direct limits and it is sufficiently small, namely $\left| \underrightarrow{lim}\,(N'/L_{i})\right| \leq \kappa$.
Then the map $N' \rightarrow M$ can be factored through a $\mathcal{W}$-injective module $\underrightarrow{lim}\,(N'/L_{i})$. Let $f \in Hom_{R}(\underrightarrow{lim}\,(N'/L_{i}), M)$. We define a map $\overline{f} \colon \underrightarrow{lim}\,(N'/L_{i}) \rightarrow \bigoplus N^{\left(Hom_{R}(N, M)\right)}$ such that $\overline{f}(n' + L') = (n_{f}, 0, 0, \cdots)$, with $n_{f} = n$. Clearly, $\overline{f}$ is a linear map. Then the following diagram
\begin{center}
\[\xymatrix@C-.15pc@R-.18pc{
&\underrightarrow{lim} N'/L_{i} \ar@{-->}[dl] \ar[d]^f &\\
\bigoplus\limits_{N \in \mathcal{I}_0} N^{\left(Hom_{R}(N, M)\right)} \ar[r]^\phi & M &
}\]
\end{center}
is commutative. Hence we get the following commutative diagram
\begin{center}
\begin{xy}
(0,0)*+{ };
(35,0)*+{\bigoplus\limits_{N \in \mathcal{I}_0} N^{\left(Hom_{R}(N, M)\right)}}="v1";%
(94,0)*+{M.}="v2";%
(66,10)*+{\underrightarrow{lim} N'/L_{i}}="v3";%
(59,30)*+{N'}="v4";%
%arrows
{\ar@{->}_{\phi} "v1"; "v2"};%
{\ar@{->}_{f} "v3"; "v2"};%
{\ar@{-->} "v4"; "v1"};%
{\ar@{->} "v4"; "v2"};%
{\ar@{->} "v4"; "v3"};%
\end{xy}
\end{center}
Thus, ${\bigoplus\limits_{N \in \mathcal{I}_0} N^{\left(Hom_{R}(N, M)\right)}} \stackrel{\phi}{\longrightarrow}M$ is a $\mathcal{W}$-injective precover.
\end{proof}

\begin{proposition}\label{3.8}
Let $R$ be a coherent ring. Then the class of all $\mathcal{W}$-injective modules is closed under direct limits.
\end{proposition}

\begin{proof}
Every pure projective $R$-module is a finitely presented $R$-module over a coherent ring. Hence the proposition follows by \cite[Lemma ~3.16]{tri1} 
\end{proof}

The following Theorem follows from Theorem \ref{3.7} and Proposition \ref{3.8}.

\begin{theorem}
Let $R$ be a coherent ring. Then every $R$-module has a $\mathcal{W}$-injective cover.
\end{theorem}

\end{document}